\documentclass[a4paper,12pt]{article}
\usepackage[utf8]{inputenc}
\usepackage[french,english]{babel}
\usepackage{csquotes}
\usepackage{geometry}
\usepackage{graphicx} 
\usepackage{float}
\usepackage{amsthm}
\usepackage{amsmath}
\usepackage{amssymb}
\usepackage{mathtools} 
\usepackage{caption}
\usepackage{subcaption}
\usepackage{appendix}
\usepackage{xfrac}
\usepackage{faktor}
\usepackage{hyperref}
\usepackage{cleveref}
\usepackage{tikz-cd}
\usepackage{stmaryrd} 
\usepackage{mathbbol}
\usepackage{enumitem}
\usepackage{xcolor}
\hypersetup{
    colorlinks=true,
    linkcolor=blue,
    filecolor=magenta, 
    citecolor=teal,
    urlcolor=cyan,
    pdftitle={Extensions-on-DNC-CHAMOUX},
    pdfpagemode=FullScreen,
    }
\usepackage[backend=biber,
style=alphabetic,
maxnames=2,
natbib=true,
abbreviate=false,
sorting=none,
giveninits=true,
autocite=superscript
]{biblatex}
\usepackage{bbding} 

\newcommand\cald[0]{\mathcal{D}}
\newcommand\cale[0]{\mathcal{E}}

\newcommand\caln[0]{\mathcal{N}}
\newcommand\cals[0]{\mathcal{S}}
\newcommand\Exp[0]{\mathcal{E}xp}

\newcommand\mr[0]{\mathbb{R}}
\newcommand\mc[0]{\mathbb{C}}
\newcommand\mz[0]{\mathbb{Z}}
\newcommand\mn[0]{\mathbb{N}}

\newcommand\dnc[0]{\operatorname{DNC}}

\DeclarePairedDelimiterX\braket[2]{\langle}{\rangle}{#1\,\delimitersize\vert\,\mathopen{}#2}

\DeclarePairedDelimiterX\dual[2]{\langle}{\rangle}{#1\,,\mathopen{}#2}

\theoremstyle{plain}
\newtheorem{theorem}{Theorem}[section]
\newtheorem{corollary}{Corollary}[theorem]
\newtheorem{lemma}[theorem]{Lemma}

\theoremstyle{definition}
\newtheorem{definition}[theorem]{Definition}
\newtheorem{prop}[theorem]{Proposition}
\newtheorem{example}[theorem]{Example}

\newtheorem{remark}[corollary]{Remark}

\begin{document}

\title{Extensions of homogeneous distributions on deformations to the normal cone}
\author{Moudrik CHAMOUX}
\date{December 6th, 2024}
\maketitle

\thispagestyle{empty}
\newpage

\newpage
\thispagestyle{empty}

\part*{Abstract}
On a deformation to the normal cone $\dnc(M,V)$ we show that given a distribution $u\in\cald'(\dnc(M,V)\setminus V\times\mr)$ if $u$ is homogeneous of order $a$ for the zoom action, then it admits an $a$-homogeneous extension $\widetilde{u}\in\cald'(\dnc(M,V))$. We describe all such extensions and discuss briefly about how it translates to the work of Van Erp and Yuncken in \cite{vanerp2017groupoid}. The technique used come from the results on the extension of weakly homogeneous distributions provided by Yves Meyer in \cite{meyer1997wavelets}.

\newpage
\thispagestyle{empty}
\tableofcontents
\endtitlepage

\section*{Introduction}\label{sec:intro maths}
\addcontentsline{toc}{section}{\nameref{sec:intro maths}}

In their two papers \cite{van_Erp_2017}, \cite{vanerp2017groupoid} Erik Van Erp and Robert Yuncken, inspired by the work of Claire Debord and Georges Skandalis in \cite{debord2013adiabatic}, managed to give a purely geometric and algebraic characterization of pseudodifferential operators on a filtered manifold. The starting idea can be traced back to the work of Connes presented in \cite{connes1994noncommutative} where he states an index theorem using the tangent groupoid 
\begin{equation*}
    \mathbb T M=M\times M\times\mr^*\sqcup TM\times\{0\}\rightrightarrows M\times\mr.
\end{equation*}
Very roughly the idea is as follows : given a manifold $M$ and an elliptic differential operator $D$ the index theorem relates the analytical index of $D$ with the topological index of $D$ (which only depends on $\sigma(D)$ the principal symbol of $D$ and topological invariants of $M$). The analytical index of $D$ can be obtained through the study of the $K$-theory of the $C^*$-algebra associated to the pair groupoid $M\times M$ whereas the topological index can be retrieved through the study of the $K$-theory of the $C^*$-algebra associated to the tangent bundle $TM$ where the cosymbol of $D$ "lives". The idea of Connes was to construct from $D$ an object $\mathbb D$ living on $\mathbb T M$ such that :
\begin{enumerate}
    \item $\mathbb D_{1}=D$,
    \item $\mathbb D_{0}$ is the cosymbol of $D$,
    \item $\mathbb D=(\mathbb D_t)_t$ is a continuous family of Fredholm operators.
\end{enumerate}
Then one can use homotopy results to obtain the index theorem by the study of the K-theory associated to the $C^*$-algebra associated to $\mathbb TM$.

In \cite{vanerp2017groupoid} the authors constructed an explicit realization of a tangent groupoid $\mathbb T_H M$ with respect to a filtration $H$ on a manifold $M$ in \cite{vanerp2017groupoid}. Then in \cite{van_Erp_2017} they develop the pseudodifferential calculus suggested by \cite{debord2013adiabatic}. Moreover Erik Van Erp and Robert Yuncken manage to characterize explicitly, differential operators and pseudodifferential operators on $M$ compatible with the filtration $H$. So in particular they retrieve (pseudo)-differential calculus in a coordinate free way.

\begin{theorem}[Characterization of differential operators on $M$ \cite{van_Erp_2017}]\ \\
A semiregular kernel $P \in \mathcal{E}'_r(M \times M)$ is the Schwartz kernel of a differential operator of order $a$ if and only if $P = \mathbb{P}|_{t=1}$ for some $\mathbb{P} \in \mathcal{E}'_r(\mathbb{T}M)$ such that 
\begin{equation*}
    \forall \lambda \in \mathbb{R}_+^\times,\ \alpha_{\lambda*}\mathbb{P} = \lambda^a \mathbb{P}.
\end{equation*}
\end{theorem}

\begin{theorem}[Characterization of pseudodifferential operators on $M$ \cite{van_Erp_2017}]\ \\
A semiregular kernel $P \in \mathcal{E}'_r(M \times M)$ is the Schwartz kernel of a properly supported classical pseudodifferential operator of order $a$ if and only if $P = \mathbb{P}|_{t=1}$ for some $\mathbb{P} \in \mathcal{E}'_r(\mathbb{T}M)$ such that 
\begin{equation*}
    \forall,\ \lambda \in \mathbb{R}_+^\times,\ \alpha_{\lambda*}\mathbb{P} - \lambda^a \mathbb{P} \text{ is a smooth properly supported density of the r-fibers.}
\end{equation*}
\end{theorem}

Here $\alpha$ denotes the "zoom action" defined in \cite{debord2013adiabatic} $\alpha : \mr_+^*\curvearrowright \mathbb T_H M$. But one should notice that the zoom action was defined in the general context of the Deformation to the Normal Cone of a submanifold $V$ in a manifold $M$ : $\dnc(M,V)$. Also note that the adiabatic groupoid of \cite{debord2013adiabatic} and so the tangent groupoid $\mathbb T M$ are instances of Deformations to the Normal Cone. Also note that the filtered tangent groupoid $\mathbb T_H M$ can be obtained with iterated Deformations to the Normal Cone thanks to \cite{Mohsen2020}. If Van Erp and Yuncken ruled out the case of r-fibred (quasi)homogeneous distributions on $\mathbb T_H M$, in this context one might ask the naive question : what do the homogeneous distributions for the zoom action look like on $\dnc(M,V)$ ? This is what this work was meant to investigate. The initial problem of this work was the following. Given a manifold $M$ and an embedded submanifold $V$, take the Deformation to the Normal Cone $\dnc(M,V)$.\\

\emph{Given an $a$-homogeneous distribution $u\in\cald'\left(\dnc(M,V)\setminus V\times\mr\right)$, is there an $a$-homogeneous extension $\widetilde{u}\in\cald'\left(\dnc(M,V)\right)$, and is it unique ?}\\

\section*{Acknowledgements}\label{sec:acknowledgements}
\addcontentsline{toc}{section}{\nameref{sec:acknowledgements}}
I would like to thank my advisors Jean-Marie Lescure and Omar Mohsen for their continuous support through the research project that lead to this paper.

\section{General results on (weakly) homogeneous distributions}\label{section general results}

This section presents general results used in the following sections. Given a manifold $M$, we will denote by $C^{-\infty}(M)$ the dual of $C^\infty_c(M)$.

\begin{definition}[Homogeneous functions, distributions]\label{def homogeneous function}\ \\
    Let $M$ be a smooth manifold. Let $\mr_+^* \overset{\alpha}{\curvearrowright}M$ be a smooth group action and let $X$ be an open subset of $M$ stable under this action. Let $f\in \mc^X$ be a function, $f$ is called "positively homogeneous of degree  $a\in\mc$" or simply "homogeneous of degree  $a\in\mc$" if:
    \begin{equation}
        \forall s >0,\ \forall x \in X,\ f(\alpha_s(x))=s^a f(x).
    \end{equation}
    A distribution $u\in  C^{-\infty}(X)$  is called "positively homogeneous of degree  $a\in\mc$" or simply "homogeneous of degree  $a\in\mc$" if:
    \begin{equation}\label{hom_distrib_dnc_charac_eq}
        \forall s >0,\ \forall \phi \in  C^{\infty}_c(X),\ \dual{\alpha_s^*u}{\phi}:=\dual{u}{\alpha_{s^{-1}}^*\phi}=s^a\dual{u}{\phi}.
    \end{equation}
    Let us denote by $H^a_\alpha(X)$ the set of such distributions. When there is no ambiguity we will denote $H^a_\alpha:=H^a_\alpha(M)$.
\end{definition}

\begin{example}\label{example dirac homogeneous}
    Let $\alpha$ be the dilation action $\mr_+^*\curvearrowright \mr$ given by $\forall s>0,\ \forall t\in\mr,\ \alpha_s(t)=st$. Given $k\in\mn$ the distribution $\partial_t^k\delta_0$ is $-k$-homogeneous.
\end{example}

\begin{remark}\label{remark direct sum homogeneous spaces}
    Observe in the above definition that given $a\neq b$, the set $H^a_\alpha(X)$ is a vector space and $H^a_\alpha(X)\cap H^b_\alpha(X)=\{0\} $.
\end{remark}

\begin{definition}[\cite{hörmanderAnalysis1} section 2.3]Let $a\in\mc$ then 
    \begin{enumerate}
        \item if $\mathfrak{Re}(a)>0$ we denote 
        \begin{equation*}
            t_+^{a-1} = [ \varphi \mapsto \displaystyle\int_{\mr_+^*} t^{a-1} \varphi(t)dt ] \in  C^{-\infty}(\mr).
        \end{equation*}
        \item If $\mathfrak{Re}(a)\in]-\infty,0]\setminus\rrbracket-\infty,0\rrbracket$ we denote 
        \begin{equation*}
            t_+^{a-1} = [ \varphi \mapsto \displaystyle\int_{\mr_+^*} t^{a-1+k} (-1)^k\frac{\varphi^{(k)}(t)}{(a-1)(a-2)\hdots (a-1+k-1)}dt ] \in  C^{-\infty}(\mr)
        \end{equation*}
        where $k=\lfloor -a \rfloor +1$.
        \item If $\mathfrak{Re}(a)\in\rrbracket-\infty,0\rrbracket$ we denote 
        \begin{equation*}
            t_+^{a-1} = [ \varphi \mapsto \displaystyle\int_{\mr_+^*} \ln(t) (-1)^k\frac{\varphi^{(k)}(t)}{(a-1)(a-2)\hdots (a-1+k-1)}dt ] \in  C^{-\infty}(\mr)
        \end{equation*}
        where $k=-a+1$.
    \end{enumerate}
    We will also use the distribution $t_-^{a-1}:=[ \varphi \mapsto \dual{t_+^{a-1}}{\varphi(-t)} ] \in  C^{-\infty}(\mr).$
\end{definition}

\begin{prop}[\cite{hörmanderAnalysis1} section 2.3]\label{prop properties of ta}\ \\
    Given $a\in\mc$, $t_+^{a-1}$ coincides with the function $t^{a-1}$ on $\mr_+^*$ and $\operatorname{supp}(t_+^{a-1})\subset\mr_+$. Moreover if $a\notin\rrbracket-\infty,0\rrbracket$, the distribution $t_+^{a-1}$ is $a$-homogeneous for the dilation action $\mr_+^*\curvearrowright \mr$. If $a\in\rrbracket-\infty,0\rrbracket$ one has:
    \begin{equation*}
        \forall s>0,\ \dual{t_+^{a-1}}{\varphi(t)}=s^{-a}\dual{t_+^{a-1}}{\varphi(st)}- \ln(s) \frac{\varphi^{(-a)}(0)}{(-a)!}.
    \end{equation*}
\end{prop}

\begin{prop}[Homogeneous distributions for the dilation action]\label{th homogeneous functions and the trivial action}\ \\
    Take any manifold $M$ and the action $\mr_+^*\overset{\alpha}{\curvearrowright} M\times\mr$ given by:
    \begin{equation*}
        \alpha = s \mapsto \begin{bmatrix}
                M\times\mr_+^* & \to & M\times\mr\\
                (p,t) & \mapsto & (p,st)
            \end{bmatrix}.
    \end{equation*}
    Let $a\in\mc$ and let $u\in C^{-\infty}(M\times\mr)$, assume $u$ to be $a$-homogeneous for the above action.
    Then there is $v\in C^{-\infty}(M)$ such that
    \begin{equation*}
        u|_{M\times\mr_+^*}=v\otimes t_+^{a-1}|_{\mr_+^*}.
    \end{equation*}
\end{prop}

\begin{proof}
    Since the distribution spaces are nuclear we only need to evaluate $u$ against tensor products $\phi\otimes\psi\in C^{\infty}_c(M)\otimes C^{\infty}_c(\mr_+^*)$ to characterize it.
    Choose $\rho\in C^{\infty}_c(\mr_+^*,\mr_+)$ such that $\displaystyle\int_{\mr_+^*} \rho(s) s^{-1}ds =1$ then 
    \begin{align*}
        \dual{u}{\phi\otimes\psi} & =\int_0^{+\infty} \rho(s)\dual{u}{\phi\otimes\psi}s^{-1}ds\\
        & =\int_0^{+\infty} s^{a-1} \dual{u_{p,t}}{\rho(s)\phi(p)\psi(st)}ds &\text{ By homogeneity of }u.\\
         & =\dual{u_{p,t}}{\int_0^{+\infty} s^{-a-1}\rho(s)\phi(p)\psi(st)ds}.
    \end{align*}
    The last equality is given by the continuity of $u$ and by the approximation of the parameter integral with a Riemann sum converges in $ C^{\infty}_c(M\times\mr_+^*)$. Finally a change of variable $s\to s t^{-1}$ yields 
    \begin{align*}
        \dual{u}{\phi\otimes\psi} & =\dual{u_{p,t}}{\int_0^{+\infty} s^{-a-1} t^{a}\rho(s t^{-1})\phi(p)\psi(s)ds}\\
        & = \int_0^{+\infty} \dual{u_{p,t}}{s^{-a} t^{a}\rho(s t^{-1})\phi(p)}\psi(s)s^{-1}ds
    \end{align*}
    and since $u$ is $a$-homogeneous 
    \begin{align*}
        \dual{u}{\phi\otimes\psi} & = \int_0^{+\infty} s^a \dual{u_{p,t}}{ t^{a}\rho(t^{-1})\phi(p)}\psi(s)s^{-1}ds.
    \end{align*}
    Now notice that $v:=[\phi \mapsto \dual{u_{p,t}}{ t^{a}\rho(t^{-1})\phi(p)} ] \in C^{-\infty}(M)$ and the proof is done.
\end{proof}

\begin{corollary}[Characterization of homogeneous distributions]\ \\
    Recall the setting of the previous proposition. We can characterize the set $H^a_\alpha(M\times\mr)$ as follows.
    \begin{itemize}
        \item If $a\in\mc\setminus\rrbracket -\infty,0\rrbracket$, then:
        \begin{equation*}
            H^a_\alpha(M\times\mr)= C^{-\infty}(M)\otimes t_+^{a-1} \oplus  C^{-\infty}(M)\otimes t_-^{a-1}.
        \end{equation*}
        \item If $a\in\rrbracket -\infty,0\rrbracket$, then:
        \begin{equation*}
            H^a_\alpha(M\times\mr)= C^{-\infty}(M)\otimes (t_+^{a-1}+(-1)^{a-1} t_-^{a-1}) \oplus  C^{-\infty}(M)\otimes \partial_t^{-a} \delta_0.
        \end{equation*}
    \end{itemize}
\end{corollary}

\begin{proof}
    Take $u\in H^a_\alpha(M\times\mr)$, thanks to \Cref{th homogeneous functions and the trivial action} 
    \begin{equation*}
        \exists v_-,v_+\in C^{-\infty}(M),\ u|_{M\times\mr^*}=v_-\otimes t_-^{a-1}|_{\mr^*} + v_+\otimes t_+^{a-1}|_{\mr^*}.
    \end{equation*}
    The above formula defines a distribution $\widetilde u$ on $M\times\mr$. One needs to compute the obstructions for it to be $a$-homogeneous. These are provided by \Cref{prop properties of ta}. And then observe that $\operatorname{supp}(u-\widetilde u)\subset M\times\{0\}$. Necessarily $u-\widetilde u=\sum_{k=1}^n v_k\otimes \partial_t^k\delta_0$ where $\forall k, v_k\in C^{-\infty}(M)$. Then since $u-\widetilde u \in H^a_\alpha$ \Cref{remark direct sum homogeneous spaces} and \Cref{example dirac homogeneous} enable us to conclude.
\end{proof}

In \cite{meyer1997wavelets}, a notion of weak homogeneity is presented for distributions on $\mr^n$. This notion can be seamlessly extended to vector bundles.

Let $V$ be a manifold, of dimension $m'$, let $W$ be a vector bundle over $V$ of rank $n$. Assume that $\pi: W\to V$ is equipped with an Euclidean metric $g$. Consider the dilation action $\alpha:  \mr_+^*\curvearrowright W$:
\begin{equation*}
    \forall s>0,\ \alpha_s=\begin{bmatrix}
        W & \to & W \\
        (p,w) & \mapsto & (p,sw)
    \end{bmatrix}.
\end{equation*}

We will denote $W\setminus\{0\}:=W\setminus V\times\{0\}$, this open set is stable by the above action.

\begin{definition}[Weakly homogeneous distributions around the base]\label{def O weakly homogeneous W}\ \\
    Given $a\in\mr$ we will denote $O^a_\alpha$ the subset of $ C^{-\infty}(W)$ of distributions $u$ such that:
    \begin{equation*}
        \forall \varphi\in C^{\infty}_c(W),\ \exists C(\varphi)>0,\ \forall s\in]0,1],\ |\dual{u}{\alpha_{s^{-1}}^*\varphi}|\leq s^{a} C(\varphi).
    \end{equation*}
    And we will denote by $\overset{\circ}{O^a_\alpha}$ the set of distributions on $W\setminus\{0\}$ satisfying the above estimate.
\end{definition}

\begin{definition}[Weakly homogeneous distributions]\label{def E weakly homogeneous W}\ \\
    Given $a\in\mr$ we will denote $E^a_\alpha$ the subset of $ C^{-\infty}(W)$ of distributions $u$ such that:
    \begin{equation*}
        \forall \varphi\in C^{\infty}_c(W),\ \exists C(\varphi)>0,\ \forall s\in\mr_+^*,\ |\dual{u}{\alpha_{s^{-1}}^*\varphi}|\leq s^{a} C(\varphi).
    \end{equation*}
    And we will denote by $\overset{\circ}{E^a_\alpha}$ the set of distributions on $W\setminus\{0\}$ satisfying the above estimate.
\end{definition}

We might omit the $_\alpha$ subscript in the above spaces when there is no ambiguity on the action.

\begin{prop}[Some properties]\label{prop some properties weak hom}
    \begin{enumerate}
        \item First notice that for $a\in\mr$, we have $E^a_\alpha\subset O^a_\alpha$ and $\overset{\circ}{E^a_\alpha}\subset \overset{\circ}{O^a_\alpha}$.
        \item Let $F\in\{E_\alpha,\overset{\circ}{E_\alpha},O_\alpha,\overset{\circ}{O_\alpha}\}$, then given a vertical vector field $X\in\Gamma(TW)$ and $a\in\mr$ we have 
        \begin{equation*}
            X(F^a) \subset F^{a-1}.
        \end{equation*}
        \item Given $b\in\mc$ and a smooth $b$-homogeneous function $f$ on the appropriate domain we have that
        \begin{equation*}
            f\cdot F^a \subset F^{a+\mathfrak{Re}(b)}.
        \end{equation*}
        \item One can easily prove that given $u\in C^{-\infty}(W)$, $F\in\{O_\alpha,\overset{\circ}{O_\alpha}\}$ and a cutoff function $\chi\in C^{\infty}(W)$ around $V$
        \begin{equation*}
            u\in F^{a} \iff \chi\cdot u F^{a}.
        \end{equation*}
        This is why a distribution in $F^{a}$ is call weakly homogeneous "around the base", it is a local assertion around $V$.
        \item Given $F\in\{O_\alpha,\overset{\circ}{O_\alpha}\}$ and two real numbers $a<b$ we have $F^b\subset F^{a}$. Note that this is not the case for the spaces of "globally" weakly homogeneous distributions of \Cref{def E weakly homogeneous W}.
    \end{enumerate}
\end{prop}

Then one can easily adapt the proof Theorem 2.1 of \cite{meyer1997wavelets} to obtain the following.
\begin{theorem}[Extension of distributions in $W\setminus\{0\}$]\label{theorem Meyers extension theorem W}\ \\
    Let $a\in\mr\setminus\rrbracket -\infty,0\rrbracket$, the restriction map:
    \begin{equation*}
       \cdot|_{W\setminus\{0\}}:  C^{-\infty}(W) \to  C^{-\infty}(W\setminus\{0\}),
    \end{equation*}
    restricts to a bijection 
    \begin{equation*}
        \cdot|_{W\setminus\{0\}}: E^a_\alpha \to \overset{\circ}{E^a_\alpha}.
    \end{equation*}
    We will denote $e_a: \overset{\circ}{E^a_\alpha} \to E^a_\alpha$ its inverse.
\end{theorem}
And we can deduce a second result from the proof of Theorem 2.1 in \cite{meyer1997wavelets}.

\begin{theorem}[Extension of locally integrable distributions in $W\setminus\{0\}$]\label{theorem locally integrable extension theorem W}\ \\
    If $a>0$, the restriction map $\cdot|_{W\setminus\{0\}}:  C^{-\infty}(W) \to  C^{-\infty}(W\setminus\{0\})$, restricts to a bijection 
    \begin{equation*}
        \cdot|_{W\setminus\{0\}}: O^a_\alpha \to \overset{\circ}{O^a_\alpha}.
    \end{equation*}
    Moreover if we denote $o_a$ its inverse and if we let $b>a$ the following diagram commutes.
    \begin{equation*}
        \begin{tikzcd}
            O^b_\alpha \arrow[r, hook] &  O^a_\alpha &  E^a_\alpha \arrow[l, hook'] \\
            \overset{\circ}{O^b_\alpha} \arrow[u, "o_b"] \arrow[r, hook] &  \overset{\circ}{O^a_\alpha}  \arrow[u, "o_a"] &  \overset{\circ}{E^a_\alpha} \arrow[l, hook']  \arrow[u, "e_a"] 
        \end{tikzcd}
    \end{equation*}
    A distribution in $O^a_\alpha$ or $\overset{\circ}{O^a_\alpha}$ for $a>0$ will be called "locally integrable around $V$" or simply "locally integrable".
\end{theorem}

\begin{remark}
    By definition, a function $f\in C^{\infty}(W\setminus\{0\})$ satisfies the Riemann integrability criterion around $V\times\{0\}$ iff in $f\omega\in \overset{\circ}{O^a}$ for some $a>0$ and $\omega \in \Omega^1(W)$, and in this case $f$ is then locally integrable. The injection $\operatorname{L^1_{loc}}(W)\cdot|\Omega^1|(W)\hookrightarrow  C^{-\infty}(W)$ shows us that the above theorem is a distributional generalization of the Riemann criterion. Hence the name of the previous theorem.
\end{remark}

Now let us introduce a few useful functional spaces with some necessary structure results.
\begin{definition}[Smooth $\pi$-properly supported functions \cite{lescure2015convolution} section 2.]\label{def smooth pi-properly supported functions}\ \\
    Let us denote $ C^{\infty}_{\operatorname{fc}}(W)$ the set of functions $\{f\in C^{\infty}(W),\ \pi:\operatorname{supp}(f)\to V \text{ is proper }\}$. Its dual is the set 
    \begin{equation*}
         C^{-\infty}_{\operatorname{bc}}(W)=\{u\in C^{-\infty}(W),\ \pi(\operatorname{supp}(u))\text{ is compact.}\}.
    \end{equation*}
    Conversely denote $ C^{\infty}_{\operatorname{bc}}(W)$ the set of functions $\{f\in C^{\infty}(W),\ \pi(\operatorname{supp}(f))\text{ is compact.}\}$. Its dual is the set 
    \begin{equation*}
         C^{-\infty}_{\operatorname{fc}}(W)=\{u\in C^{-\infty}(W),\ \pi:\operatorname{supp}(u)\to V \text{ is proper }\}.
    \end{equation*}
\end{definition}

We will denote by $S(W)$ the algebra of Schwartz functions on $W$. A definition is given in \cite{debord2013adiabatic} (Definition 1.1). Notice that the base needs not to be compact for $S(W)$ to be defined.

\begin{definition}[$\pi$-compactly supported tempered distributions]\ \\
    The space $S'(W)$ of $\pi$-compactly supported tempered distributions is the dual of the space $S(W)$ of Schwartz functions on $W$. 
\end{definition}
As indicated by the name, such distributions are $\pi$-compactly supported in the sense of \Cref{def smooth pi-properly supported functions} since $ C^{\infty}_{\operatorname{fc}}(W)\subset S(W)$ is a topological embedding.
\begin{prop}\label{prop weak hom implies tempered}
    Let $u\in  C^{-\infty}_{\operatorname{bc}}(W)$. Let $a\in\mc$ then one can prove that
    \begin{equation*}
        u\in  E^a \implies u\in S'(W).
    \end{equation*}
\end{prop}

As suggested by \Cref{prop some properties weak hom} vertical differential operators will appear in key calculations in the study of weak homogeneity.
\begin{definition}[Vertical homogeneous differential operator]\label{def homogeneous differential operator}\ \\
    Let us denote by $\operatorname{Diff_\pi}(W)$ the set of vertical differential operators. A vertical differential operator $P\in \operatorname{Diff_\pi}(W)$ is said to be homogeneous if:
    \begin{equation*}
        \exists a\in\mn,\ \forall f\in C^{\infty}(W),\ \forall s>0,\ P(\alpha_s^* f)= s^a \alpha_s^* P(f).
    \end{equation*}
    $a$ is then called the homogeneity degree of $P$. Let us denote $\operatorname{HomDiff}_\pi^a(W)$ the such operators and denote $\operatorname{HomDiff}_\pi(W):=\displaystyle\bigcup_{a\in\mn} \operatorname{HomDiff}_\pi^a(W)$ the set of all homogeneous differential operators.
\end{definition}

Now let us summon a useful characterization of tempered distributions.
\begin{lemma}[Schwartz 1966, \cite{schwartz1966théorie} Théorème VI p.239]\label{Schwartz tempered distrib lemma}
    Let $u\in C^{-\infty}(\mr^n)$, then 
    \begin{equation*}
        u\in \cals'(\mr^n) \iff \left(\exists \alpha\in\mn^n,\ \exists k\in\mn,\ \exists f\in C^0_b(\mr^n),\ u=\partial^\alpha((1+|x|^2)^k f)\right).
    \end{equation*}
\end{lemma}
This lemma can be adapted to vector bundles to obtain the following.

\begin{lemma}[Tempered distribution a a vector bundle]\label{Schwartz tempered distrib lemma vector bundle}\ \\
    Let $u\in S'(W)$, then 
    \begin{equation*}
        \exists A\in\operatorname{HomDiff}_\pi(W),\ \exists k\in\mn,\ \exists \mu \in E_\alpha^{n}\cap C^{-\infty}_{\operatorname{bc}}(W),\ u=A((1+|X|^2)^k \mu).
    \end{equation*}
\end{lemma}

\begin{proof}
    This proof will be conducted in local charts. Denote by $W\to V$ the vector bundle $\mr^{m'}\times\mr^n \to  \mr^{m'}$, and take $u\in S'(W)$. Denote by $\cals(\mr^{m'+n})$ the set of usual Schwartz functions. Since $\cals(\mr^{m'+n})\subset S(W)$ is a topological embedding, one has $S'(W)\subset \cals'(\mr^{m'+n})$. Thus we can apply \Cref{Schwartz tempered distrib lemma} to $u$, that is:
    \begin{equation*}
        \exists \beta\in\mn^{m'+n},\ \exists k\in\mn,\ \exists f\in C^0_b(\mr^{m'+n}),\ u=\partial^\beta((1+|x|^2)^k f).
    \end{equation*}
    Now let us write 
    \begin{equation*}
        u=\partial_n^{\beta_n}\partial_{m'}^{\beta_{m'}}((1+|x|^2)^k f).
    \end{equation*}
    Since $u\in C^{-\infty}_{\operatorname{bc}}(W)$, we have that $\operatorname{supp}(u)\subset B_{\mr^{m'}}(0,R)\times\mr^n$ for some $R>0$. Now denote $X$ the projection of $x\in\mr^{m'+n}$ on $\mr^n\times\{0\}$, take a cutoff function $\chi\in C^{\infty}_c(\mr^{m'})$ around $B_{\mr^{m'}}(0,R)$, and denote $\mathcal{X}=\pi^* \chi\in C^{\infty}_{\operatorname{bc}}(W)$, we have
    \begin{align*}
        u & =\mathcal{X} \cdot \partial_n^{\beta_n}\partial_{m'}^{\beta_{m'}}((1+|X|^2)^k \left( \frac{1+|x|^2}{1+|X|^2}\right)^k f)\\
        & = \partial_n^{\beta_n}((1+|X|^2)^k \mathcal{X}\partial_{m'}^{\beta_{m'}}(\left( \frac{1+|x|^2}{1+|X|^2}\right)^k f)).
    \end{align*}
    A simple integration by part shows that the distribution 
    \begin{equation*}
        \mu=\mathcal{X}\partial_{m'}^{\beta_{m'}}(\left( \frac{1+|x|^2}{1+|X|^2}\right)^k f) dx_{\mr^{m'+n}}
    \end{equation*} 
    is a sum of horizontal derivatives of functions in $C^0_b(\mr^{m'+n})\cap C^0_{\operatorname{bc}}(W)$. The \Cref{prop some properties weak hom} gives us that 
    $\mu \in E_\alpha^{n}\cap C^{-\infty}_{\operatorname{bc}}(W)$. Also in the expression 
    \begin{align*}
        u & = \partial_n^{\beta_n}((1+|X|^2)^k \mu)
    \end{align*}
    one can arbitrarily increase $\beta_n$ by integrating $(1+|X|^2)^k \mu $ vertically, thus increasing $k$ and keeping the same assumptions on $\mu$. Notice now that $\partial_n^{\beta_n}$ is a homogeneous of degree $|\beta_n|$ for the $\alpha$ action in the sense of \Cref{def homogeneous differential operator}. Now recall that $u$ has a $\pi$-compact support, the previous reasoning can be made in a finite number of local trivializations of $W$ to obtain the global result.
\end{proof}

\begin{remark}\label{remark epsilon decomposition of Schwartz functions}
    The previous result, thanks to \Cref{prop some properties weak hom}, tells us that $u\in S'(W)$ can be written a sum of weakly homogeneous distribution of integer order, namely
    \begin{equation*}
        u=A((1+|X|^2)^k \mu)=\sum_{j=1}^k \begin{pmatrix} k\\ j\end{pmatrix} A(|X|^{2j}\mu).
    \end{equation*}
    One can modify this result to obtain a decomposition in a sum of weakly homogeneous distributions of non-integer order. Recall the notations in the above proof above, in local charts we had  
    \begin{equation*}
        \mu=\mathcal{X}\partial_{m'}^{\beta_{m'}}(\left( \frac{1+|x|^2}{1+|X|^2}\right)^k f) dx_{\mr^{m'+n}}
    \end{equation*} 
    with $f\in C^0_b(\mr^{m'+n})$. Take $0<\varepsilon<1$.
    Assume $f$ to be a positive function, then define
    \begin{equation*}
        f_\varepsilon:=\min(|X|^\varepsilon,f),\ f_{-\varepsilon}:=\max(f-|X|^\varepsilon,0).
    \end{equation*}
    One easily sees that $f=f_\varepsilon+f_{-\varepsilon}$. Moreover one can prove that $f_\varepsilon dx  \in E_\alpha^{n+\varepsilon}$ and conversely $f_{-\varepsilon} dx \in E_\alpha^{n-\varepsilon}$. If $f$ is complex-valued on can apply this process to the positive and negative parts of the real and imaginary part of $f$ and still find such complex-valued $f_\varepsilon$ and $f_{-\varepsilon}$ in $C^0_b(\mr^{m'+n})$. Then one can write 
    \begin{equation*}
        \mu_\varepsilon=\mathcal{X}\partial_{m'}^{\beta_{m'}}(\left( \frac{1+|x|^2}{1+|X|^2}\right)^k f_\varepsilon) dx_{\mr^{m'+n}}
    \end{equation*}
    and 
    \begin{equation*}
        \mu_{-\varepsilon}=\mathcal{X}\partial_{m'}^{\beta_{m'}}(\left( \frac{1+|x|^2}{1+|X|^2}\right)^k f_{-\varepsilon}) dx_{\mr^{m'+n}}
    \end{equation*}
    and obtain $\mu=\mu_\varepsilon+\mu_{-\varepsilon}$. The \Cref{prop some properties weak hom} tells us that $\mu_\varepsilon \in E_\alpha^{n+\varepsilon}\cap C^{-\infty}_{\operatorname{bc}}(W)$ and $\mu_{-\varepsilon} \in E_\alpha^{n-\varepsilon}\cap C^{-\infty}_{\operatorname{bc}}(W)$. Hence 
    \begin{align*}
        u & = \partial_n^{\beta_n}((1+|X|^2)^k (\mu_\varepsilon+\mu_{-\varepsilon}))
    \end{align*}
    can be written as a sum of non-integer degree weakly homogeneous functions.
\end{remark}

\begin{remark}\label{remark weak homogeneity and temperness}
    What is most interesting in \Cref{Schwartz tempered distrib lemma} and our adaptation in the case of a fiber bundle is that given \Cref{prop weak hom implies tempered} we obtain respectively:
    \begin{equation*}
        \bigoplus_{a\in\mr} E^a(\mr^n)=\cals'(\mr^n) \text{ and }  C^{-\infty}_{\operatorname{bc}}(W)\cap \bigoplus_{a\in\mr} E^a_\alpha(W)=S'(W).
    \end{equation*}
\end{remark}

\section{The setting}\label{section 4}

Let $M$ be an $m$-dimensional smooth manifold let $V$ be an $m'$-dimensional smooth embedded submanifold of $M$. Denote $n:=m-m'$. We recall the differential description of the  $\dnc$ (Deformation to the Normal Cone) construction.

\begin{definition}[Deformation to the Normal Cone \cite{FultonMacPherson}]\label{def dnc}\ \\
    Denote $\caln_V^M$ the normal bundle to $V$ in $M$ ($\caln_V^M=TM|_V/TV\rightarrow V$). We call "Deformation to the normal Cone" to $V$ in $M$ the following set:
    \begin{equation*}
        \dnc(M,V)=M\times\mr^* \sqcup \caln_V^M \times \{0\} .
    \end{equation*}
\end{definition}

This set is endowed it with a canonical manifold structure making $\caln_V^M\times\{0\}$ and $M\times \mr^*$ into respectively closed and open submanifolds (\cite{kashiwara2002sheaves}, section 4.1).

However in practice one will work in a particular set of smooth charts and local coordinates: the exponential charts. We recall an explicit description of these charts.
Equip $M$ with a riemannian structure $(M,g)$. For every $x\in V$ the Levi-Civita connection $\nabla$ defines an exponential map 
\begin{equation*}
    \exp_x: B(x,r_x)\subset T_x M \xrightarrow{\sim} \exp_x(B(x,r_x))\subset M
\end{equation*} where $r_x>0$ is the radius of injectivity at $x$. Using that, we can find an open neighborhood $U$ of $V$ in $\caln_V^M=TV^{\perp_g}$ such that:
\begin{equation*}
    \operatorname{Exp}:= \begin{bmatrix}
        U\subset \caln_V^M & \xrightarrow{\sim} & \operatorname{Exp}(\mathcal{U})\subset M\\
        (x,X) & \mapsto & exp_x(X)
    \end{bmatrix}
\end{equation*}
is a diffeomorphism. We chose $U=\{(x,X)\in\caln_V^M,\ |X|<r_x\} $ where $|\cdot|$ denotes the metric induced by $g$ on $\caln_V^M=TV^{\perp_g}$.

\begin{definition}[Exponential Charts (\cite{HilsumSkandalis1987}, statement 3.1)]\label{def exponential charts}\ \\
    From the previous map we can define a bijection 
    \begin{equation}
        \Exp:= \begin{bmatrix}
            \mathcal{U}\subset \caln_V^M\times \mr & \rightarrow &  \Exp(\mathcal{U})\subset \dnc(M,V)\\
            (x,X,0) & \mapsto & (x,X,0)\\
            (x,X,t) & \mapsto & (exp_x(tX),t)
        \end{bmatrix}
    \end{equation}
    where $\mathcal{U}=\{(x,X,t)\in\caln_V^M,\ (x,tX)\in U\} $ and thus $  \Exp(\mathcal{U}) = \mr^*\times U \sqcup \{0\}\times \caln_V^M$. 
\end{definition}
The smooth structure on $\dnc(M,V)$ (\cite{kashiwara2002sheaves}, section 4.1) is the only one such that:
\begin{equation*}
    \begin{cases}
        \Exp: \mathcal{U}\rightarrow \dnc(M,V)\\
        \iota: M\times\mr^* \rightarrow \dnc(M,V)
    \end{cases}\text{ are embeddings.}
\end{equation*}

\begin{definition}[Zoom action \cite{debord2013adiabatic}]\ \\
    We have a Lie group action of $\mr^*_+$ on the manifold $\dnc(M,V)$ given by:
    \begin{equation}
        \forall s >0,\ \alpha_s=\begin{bmatrix}
            \dnc(M,V) & \rightarrow & \dnc(M,V)\\
            (x,X,0) & \mapsto & (x,sX,0)\\
            (x,t\neq 0) & \mapsto & (x,s^{-1}t)
        \end{bmatrix}\in \operatorname{Aut}(\dnc(M,V)).
    \end{equation}
\end{definition}

\begin{definition}[Zoom action on $\caln_V^M\times\mr$ \cite{debord2013adiabatic}]\ \\
    The zoom action $\mr_+^* \overset{\widetilde\alpha}{\curvearrowright}\caln_V^M\times\mr$ is given by:
\begin{equation}
    \forall s>0,\ \widetilde \alpha_s =\begin{bmatrix}
        \caln_V^M\times\mr & \rightarrow & \caln_V^M\times\mr\\
        (x,X,t) & \mapsto & (x,sX,s^{-1}t)
    \end{bmatrix}.
\end{equation}
Notice that both the domain and the image of the exponential charts are stable under the respective zoom actions. Moreover one can verify that this zoom action is just $\alpha$ seen in the exponential charts i.e.:
\begin{equation*}
    \forall s>0,\ \widetilde \alpha_s = \Exp^{-1} \circ \alpha_s \circ \Exp.
\end{equation*}
\end{definition}

\begin{remark}
    Note that $\alpha$ induces a free, proper, action on $\dnc(M,V)\setminus V\times\mr$ and $M\times\mr_+^*$, but this action is neither proper nor free on $\dnc(M,V)$.
\end{remark}
We would like to study homogeneous distributions for the $\alpha$-action (\Cref{def homogeneous function}) on $\dnc(M,V)\setminus V\times\mr$ and their extensions on $\dnc(M,V)$. Thanks to the localisation property of distributions, and since $V\times\mr$ is included in the range of the $\Exp$ map, the study of homogeneous distributions for the $\alpha$ action boils down to the study of homogeneous distributions for the $\widetilde \alpha$ action in the domain of $\Exp$. Indeed we have:
\begin{equation*}
\begin{matrix}
    \Exp_*:  \begin{matrix}
        C^{-\infty}(\mathcal{U}) & \xrightarrow[\text{tvs iso}]{\sim} & C^{-\infty}(\Exp(\mathcal{U})) \\
        \begin{Bmatrix}
            \alpha\text{-homogeneous}\\
            \text{distributions}
        \end{Bmatrix} & \xrightarrow[]{1:1} & \begin{Bmatrix}
            \widetilde\alpha\text{-homogeneous}\\
            \text{distributions}
        \end{Bmatrix}
    \end{matrix}.
\end{matrix}
\end{equation*}
Thus in the following we will assume without loss of generality that the exponential charts to be global, that is $\mathcal{U}=\caln_V^M\times\mr$ and $\Exp(\mathcal{U})= \dnc(M,V)$.

By abuse of notation we will denote $\alpha$ the zoom action on $\caln_V^M\times\mr$. Notice that the action $\mr_+^* \overset{\alpha}{\curvearrowright}\caln_V^M\times\mr$ induces an action on each fiber $\caln_x \times\mr$ such that for any $f\in\mc^{\caln_V^M\times\mr}$, $f$ is homogeneous iff:
\begin{equation*}
    \forall x\in V,\ \forall (X,t)\in\caln_x \times\mr,\ f(x,sX,s^{-1}t)=s^af(x,X,t).
\end{equation*}
So $f$ is homogeneous iff for any $x\in V$, $f(x,\cdot)$ is homogeneous on $\caln_x \times\mr$.

\begin{lemma}[Derivations of homogeneous functions]\ \\
    Let $f$ be a smooth function on $\caln_V^M\times\mr$. Let $a\in\mc$. If $f$ is $a$-homogeneous then:
    \begin{enumerate}
        \item $\partial_t f$ is $(a+1)$-homogeneous,
        \item For any $x\in V$, for any $X\in\caln_x$ the function $X f(x,\cdot)$ is $(a-1)$-homogeneous.
    \end{enumerate}
\end{lemma}

\begin{proof}
    Because derivations are local notions we can work as in a local trivialization of $\caln_V^M \times \mr\to V$. Fix $x$ in $V$, let $(X,t)\in\caln_x \times\mr$. Chose an arbitrary basis $(e_i)_{i=1}^n$ of $\caln_x$ and do the computations with respect to it without loss of generality.
    \begin{enumerate}
        \item Since $f$ is $a$-homogeneous we have $f(x,X,t)=t^{-a}f(x,tX,1)$. Notice how any function of $(X,t)$ of the form $(X,t)\mapsto g(tX)$ is $0$-homogeneous. Now compute $\partial_t f(x,X,t)$ using the Leibnitz rule:
        \begin{align*}
            \partial_t f(x,X,t) & = -a t^{-a-1}f(x,tX,1) +t^{-a}\displaystyle\sum_{i=1}^n X_i (\partial_{Xi} f)(x,tX,1)
        \end{align*}
        which is visibly $(a+1)$-homogeneous.
        \item Let $i\in\llbracket1,n\rrbracket$. Compute $\partial_{Xi} f(x,X,t)$ with the chain rule:
        \begin{align*}
            \partial_{Xi} f(x,X,t) & = \partial_{Xi}(t^{-a}f(x,tX,1))\\
            & = t^{-a}\partial_{Xi}(f(x,tX,1))\\
            & = t^{-a+1}(\partial_{Xi}f)(x,tX,1)
        \end{align*}
        which is visibly $(a-1)$-homogeneous.
    \end{enumerate}
\end{proof}

We will start with results on smooth functions as a guiding example.

\begin{theorem}[Characterization of smooth homogeneous functions on $\dnc(M,V)$]\label{th charac smooth hom functions}\ \\
    Let $f\in C^{\infty}(\dnc(M,V))$ be an $a$-homogeneous function, then there are three cases.
    \begin{enumerate}
    \item If $a\in\rrbracket-\infty,0\rrbracket$, $f$ is of the form:
    \begin{equation*}
        (-1)^a f_{M-}\otimes t_-^{-a} + f_{M+}\otimes t_+^{-a}:=\begin{bmatrix}
            \dnc(M,V) & \rightarrow & \mc\\
            (x,t) & \mapsto & (-1)^a f_{M-}(x)t_-^{-a} + f_{M+}(x)t_+^{-a} \\
            (x,X,0) & \mapsto & \begin{cases}
                0 & \text{ if } a<0\\
                f_{M+}(x) & \text{ if } a=0\\
            \end{cases}
        \end{bmatrix}
    \end{equation*}
    where $f_{M-}$ and $f_{M+}$ are smooth functions on $M$  such that $f_{M+}-f_{M-}$ is flat on $V$ (i.e. $f_{M-}$, $f_{M+}$ and all their derivatives of any order coincide on $V$).
    \item If $a\in\mn^*$, $f$ is of the form:
    \begin{equation*}
        (-1)^a f_{M-}\otimes t_-^{-a} + f_{M+}\otimes t_+^{-a}:=\begin{bmatrix}
            \dnc(M,V) & \rightarrow & \mc\\
            (x,t) & \mapsto & (-1)^a f_{M-}(x)t_-^{-a} + f_{M+}(x)t_+^{-a} \\
            (x,X,0) & \mapsto & \displaystyle\frac{1}{a!}(X^{(a)} f_{M+})(x)
        \end{bmatrix}
    \end{equation*}
    where $f_{M-}$ and $f_{M+}$ are smooth functions on $M$ that vanish at order $a-1$ on $V$ (i.e. $f_{M-}$, $f_{M+}$ and all their derivatives up to order $a-1$ order vanish on $V$) and such that $f_{M+}-f_{M-}$ is flat on $V$. Here $X^{(a)}\cdot$ denotes the $a$-th directional derivative along the vector $X$.
    \item If $a\in\mc\setminus\mz$, then $f$ is of the form:
    \begin{equation*}
        f_{M-}\otimes t_-^{-a} + f_{M+}\otimes t_+^{-a}:=\begin{bmatrix}
            \dnc(M,V) & \rightarrow & \mc\\
            (x,t) & \mapsto & f_{M-}(x)t_-^{-a} + f_{M+}(x)t_+^{-a} \\
            (x,X,0) & \mapsto & 0
        \end{bmatrix}
    \end{equation*}
    where $f_{M-}$ and $f_{M+}$ are smooth functions on $M$ that are flat on $V$.
    \end{enumerate}
\end{theorem}

\begin{proof}
    Let $a\in\mc$. Let $f\in C^{\infty}(\dnc(M,V))$ be an $a$-homogeneous function on $\dnc(M,V)$. Thanks to \Cref{th homogeneous functions and the trivial action} we have
    \begin{equation*}
        \exists f_{M-}, f_{M+}\in C^{\infty}(M),\ f_{M\times\mr_+^*}=f_{M-}\otimes t_-^{-a} + f_{M+}\otimes t_+^{-a}.
    \end{equation*}
    Denote 
    \begin{equation*}
        g:=f_{M\times\mr^*}\circ\Exp,
    \end{equation*}
    for the sake of simplicity we can assume $U=\caln_V^M$ in the following proof. Then $g$ is a smooth function on $\caln_V^M\times\mr^*_+$ and $a$-homogeneous for the $\alpha$-action.
    Denote:
    \begin{equation*}
        \gamma_+=\begin{bmatrix}
            \caln_V^M & \rightarrow & \mc \\
            (x,X) & \mapsto & g(x,X,1)
        \end{bmatrix}\text{ and }\gamma_-=\begin{bmatrix}
            \caln_V^M & \rightarrow & \mc \\
            (x,X) & \mapsto & g(x,X,-1)
        \end{bmatrix}.
    \end{equation*}
    Then 
    \begin{equation*}
        g=\begin{bmatrix}
            \caln_V^M\times\mr^* & \rightarrow & \mc \\
            (x,X,t) & \mapsto & (-1)^a\gamma_-(x,tX)t_-^{-a}+\gamma_+(x,tX)t_+^{-a}
        \end{bmatrix}.
    \end{equation*}
    Let $\widetilde g$ be a smooth extension of $g$ on $\caln_V^M\times\mr$. Let  us look at the derivatives of $g$ on the $\mr$ direction.
    \begin{equation}\label{eq g partial t}
        \forall t>0,\ \forall l\in\mn,\ (\partial_t^l g)(x,X,t)=\sum_{k=0}^l \begin{pmatrix}
            l\\
            k
        \end{pmatrix} \begin{pmatrix}
            -a\\
            l-k
        \end{pmatrix}(l-k)! t^{-a-l+k} (X^{(k)}\gamma_+)(x,tX),
    \end{equation}
    where $X^{(l-k)}\gamma_+$ denotes the $(l-k)$-th directional derivative of $\gamma_+$ in the direction of $X$ and 
    \begin{equation*}
        \begin{pmatrix}
            -a\\
            l-k
        \end{pmatrix}=\begin{cases}
            \displaystyle\frac{(-a)(-a-1)\hdots(-a-l+k+1)}{(l-k)!} & \text{ if }k<l \\
            1 & \text{ if }k=l 
        \end{cases}
    \end{equation*}
    is the generalized binomial coefficient.
    If $\widetilde g$ extends $g$, then
    \begin{equation*}
        \forall l\in\mn,\ (\partial_t^l g)(x,X,t)\xrightarrow[t\rightarrow 0^+]{}(\partial_t^l \widetilde{g})(x,X,0).
    \end{equation*}
    Using Taylor's formula for $t\mapsto (X^{(k)}\gamma_+)(x,tX)$ we obtain:
    \begin{equation*}
        \forall l\in\mn,\ (\partial_t^l \widetilde{g})(x,X,t)=\sum_{k=0}^l \begin{pmatrix}
            l\\
            k
        \end{pmatrix} \begin{pmatrix}
            -a\\
            l-k
        \end{pmatrix}(l-k)! t^{-a-l+k} \sum_{i=0}^{l+\lceil a\rceil -k}\frac{t^i}{i!}(X^{(k+i)}\gamma_+)(x,0) +o_{t\rightarrow 0^+}(1).
    \end{equation*}
    Now using some usual reindexation tricks and the generalized Vandermonde identity that is:
    \begin{align*}
        \sum_{k=0}^l \begin{pmatrix}
            i\\
            k
        \end{pmatrix}\begin{pmatrix}
            -a\\
            l-k
        \end{pmatrix}=\begin{pmatrix}
            i-a\\
            l
        \end{pmatrix}
    \end{align*}
    one can retrieve
    \begin{align*}
        (\partial_t^l \widetilde{g})(x,X,t) & = \sum_{i=0}^{\max(l,l+\lceil a \rceil)} \frac{l!}{i!} t^{i-a-l} (X^{(i)}\gamma_+)(x,0) \begin{pmatrix}
            i-a\\
            l
        \end{pmatrix} +o_{t\rightarrow 0^+}(1).
    \end{align*}
    Now the above sum exhibits different behaviours depending on the value of $a$. Assume $a\in\rrbracket-\infty,0\rrbracket$. Then 
    \begin{equation*}
        \forall i \in \llbracket 0, l+a\llbracket,\ \begin{pmatrix}
            i-a\\
            l
        \end{pmatrix} = 0.
    \end{equation*}
    So we have 
    \begin{align*}
        (\partial_t^l \widetilde{g})(x,X,t) & = \sum_{i=\max(0,l+a)}^{l} \frac{l!}{i!} t^{i-a-l} (X^{(i)}\gamma_+)(x,0) \begin{pmatrix}
            i-a\\
            l
        \end{pmatrix} +o_{t\rightarrow 0^+}(1)
    \end{align*}
    which yields
    \begin{equation*}
        (\partial_t^l \widetilde{g})(x,X,t) = \begin{cases}
            0 +o_{t\rightarrow 0^+}(1) & \text{if } l<-a \\
            \frac{l!}{(l+a)!} (X^{(l+a)}\gamma_+)(x,0) +o_{t\rightarrow 0^+}(1) & \text{if } l\geq-a
        \end{cases}.
    \end{equation*}
    And for negative values of $t$ we have:
    \begin{equation*}
        (\partial_t^l \widetilde{g})(x,X,t) = \begin{cases}
            0 +o_{t\rightarrow 0^-}(1) & \text{if } l<-a \\
            \frac{l!}{(l+a)!} (X^{(l+a)}\gamma_-)(x,0) +o_{t\rightarrow 0^+-}(1) & \text{if } l\geq-a
        \end{cases}.
    \end{equation*}
    So if $g$ admits such a smooth extension $\widetilde{g}$, necessarily all the derivatives at every order of $\gamma_+$ and $\gamma_-$ coincide on $V$. Conversely such condition is sufficient given that for $I\in\mn^n$ and any $J\in\mn^n$
    \begin{equation*}
        \partial_X^I\partial_x^J g=\begin{bmatrix}
            \caln_V^M\times\mr^* & \rightarrow & \mc \\
            (x,X,t) & \mapsto & (-1)^{a-|I|}(\partial_X^I\partial_x^J\gamma_-)(x,tX)t_-^{-a+|I|}+(\partial_X^I\partial_x^J\gamma_+)(x,tX)t_+^{-a+|I|}
        \end{bmatrix}
    \end{equation*}
    is an $a-|I|$-homogeneous function the previous reasoning applies for the limit of the derivatives with respect to $t$ at zero. This observation concludes the proof in this case.
    Assume now that $a\in\mn^*$
    
    \begin{align*}
        (\partial_t^l \widetilde{g})(x,X,t) = & \sum_{i=0}^{l+a} \frac{l!}{i!} t^{i-a-l} (X^{(i)}\gamma_+)(x,0) \begin{pmatrix}
            i-a\\
            l
        \end{pmatrix} +o_{t\rightarrow 0^+}(1)\\
        = & \sum_{i=0}^{a-1} \frac{l!}{i!} t^{i-a-l} (X^{(i)}\gamma_+)(x,0) \begin{pmatrix}
            i-a\\
            l
        \end{pmatrix}\\
        & + \sum_{i=0}^{l} \frac{l!}{(i+a)!} t^{i-l} (X^{(i+a)}\gamma_+)(x,0) \begin{pmatrix}
            i\\
            l
        \end{pmatrix} +o_{t\rightarrow 0^+}(1).
    \end{align*}
    Now since $\forall k\in \llbracket i+1, +\infty\llbracket,\ \begin{pmatrix}
            i\\
            k
        \end{pmatrix}=0$ we have:
    \begin{align*}
        (\partial_t^l \widetilde{g})(x,X,t) = & \sum_{i=0}^{a-1} \frac{l!}{i!} t^{i-a-l} (X^{(i)}\gamma_+)(x,0) \begin{pmatrix}
            i-a\\
            l
        \end{pmatrix} + \frac{l!}{(l+a)!} (X^{(l+a)}\gamma_+)(x,0)  +o_{t\rightarrow 0^+}(1).
    \end{align*}
    And since 
    \begin{equation*}
        \forall i\in\llbracket 0, a-1\rrbracket,\ \begin{pmatrix}
            i-a\\
            l
        \end{pmatrix}\neq 0,
    \end{equation*}
    for $(\partial_t^l \widetilde{g})(x,X,t)$ to admit a limit when $t\rightarrow 0^+$ we need 
    \begin{equation*}
        \forall i\in\llbracket 0, a-1\rrbracket,\ (X^{(i)}\gamma_+)(x,0)=0.
    \end{equation*}
    Now given that for $I\in\mn^n$ and any $J\in\mn^n$
    $\partial_X^I\partial_x^J g$ is an $a-|I|$-homogeneous function the previous reasoning applies for the limit of the derivatives with respect to $t$ at zero. Then we have that $g$ admits a smooth extension on $\caln_V^M\times\mr$ if and only if all the derivatives at every order of $\gamma_+$ and $\gamma_-$ coincide on $V$ and both functions vanish at order $a-1$ on $V$. Now assume $a\in\mc\setminus\mz$. We still have
    \begin{equation*}
        (\partial_t^l \widetilde{g})(x,X,t) = \sum_{i=0}^{l+\lceil a\rceil} \frac{l!}{i!} t^{i-a-l} (X^{(i)}\gamma_+)(x,0) \begin{pmatrix}
            i-a\\
            l
        \end{pmatrix} +o_{t\rightarrow 0^+}(1).
    \end{equation*}
    Then since 
    \begin{equation*}
        \forall i\in\llbracket 0, l+\lceil a\rceil\rrbracket,\ \begin{pmatrix}
            i-a\\
            l
        \end{pmatrix}\neq 0
    \end{equation*}
    this expression admits a limit when $t\rightarrow 0^+$ if and only if 
    \begin{equation*}
        \forall i\in\llbracket 0, l+\lceil a\rceil\rrbracket,\ (X^{(i)}\gamma_+)(x,0)=0.
    \end{equation*}
    This means that if $g$ admits a smooth extension on $\caln_V^M\times\mr$ then if all the derivatives at every order of $\gamma_+$ and $\gamma_-$ vanish on $V$. Now given any $I\in\mn^n$ and any $J\in\mn^n$
    $\partial_X^I\partial_x^J g$ is an $a-|I|$-homogeneous function, the previous reasoning applies for the limit of the derivatives with respect to $t$ at zero. Thus $g$ admits a smooth extension on $\caln_V^M\times\mr$ if and only $\gamma_+$ and $\gamma_-$ are flat on $V$.
    This concludes the proof.
\end{proof}

\section{On the extension of distributions}\label{section 7}
Recall the setting we had on $\caln_V^M\times\mr$ with the $\alpha$ action:
\begin{equation*}
    \forall s>0,\ \alpha_s=\begin{bmatrix}
        \caln_V^M\times\mr & \to & \caln_V^M\times\mr \\
        (x,X,t) & \mapsto & (x,sX,s^{-1}t)
    \end{bmatrix}.
\end{equation*}
Define the action of dilatation along the normal fibers $\mr_+^*\overset{\beta}{\curvearrowright}\caln_V^M$ by:
\begin{equation*}
    \forall s>0,\ \beta_s=\begin{bmatrix}
        \caln_V^M & \to & \caln_V^M \\
        (x,X) & \mapsto & (x,sX)
    \end{bmatrix}.
\end{equation*}
By abuse of notation denote $\beta$ the induced action on $\caln_V^M\times\mr$ (leaving the real line invariant). Denote as well the action of dilatation along the time coordinate $\mr_+^*\overset{\gamma}{\curvearrowright}\mr$ by:
\begin{equation*}
    \forall s>0,\ \gamma_s=\begin{bmatrix}
        \mr & \to & \mr \\
        t & \mapsto & st
    \end{bmatrix}.
\end{equation*}
By abuse of notation denote $\gamma$ the induced action on $\caln_V^M\times\mr$ (leaving the normal bundle invariant). Then by definition 
\begin{equation*}
    \forall s>0,\ \alpha_s=\beta_s\circ \gamma_{s^{-1}},
\end{equation*}
notice that the $\gamma$ and $\beta$ actions easily commute. When homogeneity will be mentioned in the following without any precision it will mean homogeneity with respect to $\alpha$.

Let $U$ be an open subset of $\caln_V^M\times\mr$ stable by the zoom action, let $T\in C^{\infty}_c(U)$ be a distribution and $a\in\mc$. Then $T$ is $a$-homogeneous iff:
\begin{equation*}
    \forall \varphi\in C^{\infty}_c(U),\ \forall s>0,\ \dual{T}{\varphi}= s^a \dual{T}{\alpha_s^*\varphi}
\end{equation*}
which means
\begin{equation}\label{homogeneity links section 7}
    \forall \varphi\in C^{\infty}_c(U),\ \forall s>0,\ \dual{T}{\gamma_s^*\varphi}= s^a \dual{T}{\beta_s^*\varphi}.
\end{equation}
This identity means that even if the $a$-homogeneity condition does not prescribes itself a behaviour around the submanifolds $V\times\mr$ or $\caln_V^M\times\{0\}$ however it links exactly the behaviour of $T$ around $V\times\mr$ or $\caln_V^M\times\{0\}$. In this sense we can make a use of the results of \Cref{section general results}.

\begin{remark}[Notations]
    Recall the weak-homogeneity spaces of \Cref{section general results}. We will use them on $\caln_V^M\times\mr\to V\times\mr$ and on $\caln_V^M\times\mr\to \caln_V^M$ with respectively a $\beta$ and a $\gamma$ subscript.
\end{remark}

First we can easily adapt local integrability around $V\times\mr$ and $\caln_V^M\times\{0\}$ simultaneously.

\begin{theorem}[Locally integrable homogeneous distributions]\label{th locally integrable homogeneous distributions NxR}\ \\
    Let $b>0$ and let $a\in\mc$ such that $b-\mathfrak{Re}(a)>0$. Let $T\in H^a_\alpha((\caln_V^M\setminus V)\times\mr^*)\cap \overset{\circ}{O_\beta^b}$.
    Then $T$ admits an extension $\dot{T}\in H^a_\alpha(\caln_V^M\times\mr)\cap O_\beta^b\cap O_\gamma^{b-\mathfrak{Re}(a)}$. Such an extension is unique.
\end{theorem}

\begin{proof}
    Let $T\in H^a_\alpha((\caln_V^M\setminus V)\times\mr^*)\cap\overset{\circ}{O_\beta^b}$. The assumption $T\in \overset{\circ}{O_\beta^b}$ means that $T$ is locally integrable in the sense of \Cref{theorem locally integrable extension theorem W} around the base $V\times\mr^*$ of the vector bundle $\caln_V^M\times\mr^*$. Thus $T$ admits a canonical extension $\overset{\circ}{T}\in C^{-\infty}(\caln_V^M\times\mr^*)\cap O_\beta^b$. Moreover, thanks to \Cref{homogeneity links section 7}, this extension satisfies $\overset{\circ}{T}\in \overset{\circ}{O}^{b-\mathfrak{Re}(a)}_\gamma$. Thus $\overset{\circ}{T}$ admits a canonical extension $\dot{T}\in C^{-\infty}(\caln_V^M\times\mr)\cap O_\gamma^{b-\mathfrak{Re}(a)}$. 
    Now take a cutoff function $\chi \in C^{\infty}(\mr_+)$ around $0$. One can easily see that $\dot{T}$ can be written as a simple limit of $a$-homogeneous distributions, namely:
    \begin{equation*}
        \dot{T}=\lim_{n\to\infty} \chi(n|tX|)\cdot T.
    \end{equation*}
    Thus $\dot{T}$ is $a$-homogeneous.
\end{proof}

Now notice that the weak homogeneity extension theorems (\Cref{theorem Meyers extension theorem W} and \Cref{theorem locally integrable extension theorem W}) does not give an explicit formula of the extension, thus giving no tool to check for the $a$-homogeneity of the extension in general. In the previous proof the local integrability assumption enables to circumvent this issue as such an extension only depends on the values of the original distribution on our open set $(\caln_V^M\setminus V)\times\mr^*$. Even if we cannot leverage the weak homogeneity property as simply it will help us towards a general answer to our question:
given an $a$-homogeneous distribution $u\in  C^{-\infty}((\caln_V^M\setminus V)\times \mr)$, is there an $a$-homogeneous distribution $\tilde{u}\in  C^{-\infty}(\caln_V^M\times \mr)$ ? And could this extension be canonical ?

First we need a polar coordinate decomposition along the normal bundle fibers on $(\caln_V^M\setminus V) \times \mr$. Denote by $\cals_V^M$ is the sphere bundle associated to $\caln_V^M$. Consider the diffeomorphism 
\begin{equation*}
    \text{\raisebox{-1 pt}{\FiveStarCenterOpen}} =\begin{bmatrix}
    \mr_+^*\times \cals_V^M \times \mr & \to &  (\caln_V^M\setminus V)\times \mr \\
    (r, \omega, t) & \mapsto & (r\omega, t)
    \end{bmatrix}.
\end{equation*}
Then for $s>0$ we have 
\begin{equation*}
    \text{\raisebox{-1 pt}{\FiveStarCenterOpen}}^{-1}\beta_s \text{\raisebox{-1 pt}{\FiveStarCenterOpen}} =\begin{bmatrix}
    \mr_+^*\times \cals_V^M \times\mr & \to & \mr_+^*\times \cals_V^M \times \mr \\
    (|X|, X/|X|, t) & \mapsto & (s|X|, X/|X|, t)
    \end{bmatrix},
\end{equation*}
which means that $\beta$ reads as the canonical radial dilation in polar coordinates. By abuse of notation we will also denote by $\beta_s$ the previous map. We will denote accordingly by $\gamma$ and $\alpha$ the corresponding actions in this setting. We can extend easily $\beta$, $\gamma$ and $\alpha$ to $\mr_+\times \cals_V^M \times\mr$ by the obvious formula.

Recall that we have a submersion extending \raisebox{-1 pt}{\FiveStarCenterOpen} and collapsing the sphere bundle $\{0\}\times \cals_V^M \times\mr$ to $V \times\mr $ given by:
\begin{equation*}
    \text{\raisebox{-1 pt}{\FiveStar}} =\begin{bmatrix}
    \mr_+\times \cals_V^M \times \mr & \to &  \caln_V^M\times \mr \\
    (r, \omega, t) & \mapsto & (r\omega, t)
    \end{bmatrix}.
\end{equation*}
Now our object of study for the following theorem will be $a$-homogeneous distributions for the $\alpha$ action on $\mr_+^* \times \cals_V^M \times\mr$ and the following theorem will be arguably our most important one. First we need a convenient lemma.
\begin{lemma}[Characterization of homogeneous distributions on $\mr_+^*\times \cals_V^M \times \mr$]\label{lemma characterization of homogeneous distributions on R+* x N x R}\ \\
    Let $u\in C^{-\infty}(\mr_+^*\times \cals_V^M \times \mr)$. Then $u$ is $a$-homogeneous iff 
    \begin{align*}
        \exists w \in  C^{-\infty}(\cals_V^M \times \mr),\ & \forall (f,g) \in C^{\infty}_c(\mr_+^*)\times C^{\infty}_c(\cals_V^M\times\mr), \\
        \dual{u}{f\otimes g} = & \int_{\mr_+^*} s^{a-1} f(s) \dual{w}{\gamma_{s^{-1}}^* g} ds.
    \end{align*}
\end{lemma}

\begin{proof}
    Let $u\in C^{-\infty}(\mr_+^*\times \cals_V^M \times \mr)$ be an $a$-homogeneous distribution. Take $(f,g) \in C^{\infty}_c(\mr_+^*)\times C^{\infty}_c(\cals_V^M\times\mr)$.
    Then 
    \begin{equation*}
        \forall s>0,\ \dual{u}{f\otimes g}=s^a \dual{u}{\alpha_{s}^* (f\otimes g)}.
    \end{equation*}
    Now choose $\rho\in C^{\infty}_c(\mr_+^*,\mr_+^*)$ such that $\int_{\mr_+^* } \rho(s)s^{-1}ds=1$. Now by integrating the last equation we obtain:
    \begin{align*}
        \dual{u}{f\otimes g} & =\int_{\mr_+^*} \rho(s)s^{a-1} \dual{u}{\alpha_{s}^* (f\otimes g)}ds\\
        & =\int_{\mr_+^*} \rho(s)s^{a-1} \dual{u_{r,\omega,t}}{f(sr)g(\omega,s^{-1} t)}ds.
    \end{align*}
    Now we can exchange between $u$ and the integral for the same reasons as in the proof of \Cref{th homogeneous functions and the trivial action} and get 
    \begin{align*}
        \dual{u}{f\otimes g} & =\dual{u_{r,\omega,t}}{\int_{\mr_+^*} \rho(s)s^{a-1} f(sr)g(\omega,s^{-1} t)ds}.
    \end{align*}
    A change of variable yields 
    \begin{align*}
        \dual{u}{f\otimes g} & =\dual{u_{r,\omega,t}}{\int_{\mr_+^*} \rho(r^{-1} s)s^{a-1} r^{-a} f(s)g(\omega,s^{-1} rt)ds}\\
        & =\int_{\mr_+^*}s^{-1}f(s) \dual{u_{r,\omega,t}}{\rho(r^{-1} s)s^{a} r^{-a} g(\omega,s^{-1} rt)}ds.
    \end{align*}
    Finally the $a$-homogeneity of $u$ gives 
    \begin{align*}
        \dual{u}{f\otimes g} &=\int_{\mr_+^*}s^{a-1}f(s) \dual{u_{r,\omega,t}}{\rho(r^{-1})r^{-a} g(\omega,s^{-1}rt)}ds
    \end{align*}
    which concludes the proof. The reciprocal statement is easily checked to be true.
\end{proof}

\begin{theorem}[Canonical extensions in polar coordinates and weak homogeneity]\label{th canonical extensions in polar coordinates and weak homogeneity}\ \\
    Take $u\in H^a_\alpha (\mr_+^*\times \cals_V^M \times \mr)$ and assume that $u\in E_\gamma^c$ for some $c\in\mr$. If $c+\mathfrak{Re}(a)\notin -\mn$ then $u$ admits a canonical $a$-homogeneous extension $\widetilde{u}\in H^a_\alpha (\mr_+\times \cals_V^M \times \mr)\cap E_\beta^{c+a}$.
\end{theorem}

\begin{proof}
    Let $u\in C^{-\infty}(\mr_+^*\times \cals_V^M \times \mr)$ be an $a$-homogeneous distribution. Thanks to \Cref{lemma characterization of homogeneous distributions on R+* x N x R} we can write:
    \begin{equation}\label{eq spherical a homogeneous}
        \dual{u}{f\otimes g} = \int_{\mr_+^*} s^{a-1} f(s) \dual{w}{\gamma_{s^{-1}}^* g} ds.
    \end{equation}
    A naive way to extend $u$ to $\mr_+\times \cals_V^M \times \mr$ would be for the previous formula to make sense for $f\in C^{\infty}_c(\mr_+)$. Now recall that by assumption 
    \begin{equation*}
        \exists C>0,\ \forall s>0,\ |s^{a-1} \dual{w}{\gamma_{s^{-1}}^* g}|\leq Cs^{c+\mathfrak{Re}(a)}
    \end{equation*}
    Then if $c+\mathfrak{Re}(a)>0$ by assumption the integrand of \Cref{eq spherical a homogeneous} is Riemann-integrable. Then 
    \begin{equation*}
        \Tilde{u}=\begin{bmatrix}
             C^{\infty}_c(\mr_+\times \cals_V^M \times \mr) & \to & \mc \\
            f(r)\otimes g(\omega,t) & \mapsto & \displaystyle\int_{\mr_+^*} s^{a-1} f(s) \dual{w}{\gamma_{s^{-1}}^* g} ds
        \end{bmatrix}
    \end{equation*}
    defines a continuous map and extends $u$ on $\mr_+\times \cals_V^M \times \mr$.
    Now assume $c+\mathfrak{Re}(a)\in ]-\infty,0]\setminus\rrbracket-\infty,0\rrbracket$. Denote $k=-\lfloor c+\mathfrak{Re}(a) \rfloor$. In this setting, if we have no guaranty of the integrability of the integrand as in \Cref{eq spherical a homogeneous}. However the estimate 
    \begin{equation}\label{eq estimate weak hom}
        \exists C>0,\ \forall s>0,\ |s^{a-1} \dual{w}{\gamma_{s^{-1}}^* g}|\leq C s^{c+\mathfrak{Re}(a)}
    \end{equation}
    still holds.
    Now observe that for $f\in C^{\infty}_c(\mr_+^*)$, by integration by parts, \Cref{eq spherical a homogeneous} can be written:
    \begin{equation*}
        \dual{u}{f\otimes g} = -\int_{\mr_+^*} f'(s) \int_{+\infty}^s t^{a-1} \dual{w}{\gamma_{t^{-1}}^* g}dt ds.
    \end{equation*}
    Note that the integral $\int_{+\infty}^s \tau^{a-1} \dual{w}{\gamma_{\tau^{-1}}^* g}d\tau$ makes sense thanks to the estimate \Cref{eq estimate weak hom}. More generally we can integrate (from infinity) up to $k$ times this estimate and obtain:
    \begin{equation}\label{eq integrated estimate weak hom}
        \begin{split}
            & \forall j\in\llbracket 0, k\rrbracket,\ \exists C_j>0,\ \forall s>0,\\
            & \quad \quad \quad \int_{+\infty}^{s} \hdots \int_{+\infty}^{t_3} \int_{+\infty}^{t_2} t_1^{a-1} \dual{w}{\gamma_{t_1^{-1}}^* g}dt_1 dt_2 \hdots dt_j \leq C_j s^{c+\mathfrak{Re}(a)+j}.
        \end{split}
    \end{equation}
    We obtain then that for $f\in C^{\infty}_c(\mr_+^*)$,
    \begin{equation*}
        \dual{u}{f\otimes g} = (-1)^k \int_{\mr_+^*} f^{(k)}(s) \int_{+\infty}^{s} \hdots \int_{+\infty}^{t_3} \int_{+\infty}^{t_2} t_1^{a-1} \dual{w}{\gamma_{t_1^{-1}}^* g}dt_1 dt_2 \hdots dt_k ds.
    \end{equation*}
    But now the above formula still makes sense if $f\in C^{\infty}_c(\mr_+)$ since the integrand is Riemann integrable, thanks to \Cref{eq integrated estimate weak hom} (take $j=k=-\lfloor c+\mathfrak{Re}(a) \rfloor$). This formula easily defines an extension $\dot{u}\in C^{-\infty}(\mr_+\times \cals_V^M \times \mr)$ of $u$. Since every integration is performed from infinity one can easily check that this extension is $a$-homogeneous.
\end{proof}

Now notice that this theorem gives us a canonical extension theorem on $(\caln_V^M\setminus V) \times \mr$.
Indeed we have the polar coordinates isomorphism inducing:
\begin{equation*}
    \text{\raisebox{-1 pt}{\FiveStarCenterOpen}}_* = C^{-\infty}(\mr_+^*\times \cals_V^M \times \mr) \xrightarrow[]{\sim}  C^{-\infty}((\caln_V^M\setminus V)\times \mr) 
\end{equation*}
and the collapsing extension inducing 
\begin{equation*}
    \text{\raisebox{-1 pt}{\FiveStar}}_* = C^{-\infty}(\mr_+\times \cals_V^M \times \mr) \xrightarrow[]{}  C^{-\infty}(\caln_V^M\times \mr).
\end{equation*}
This, given $c\in\mr$ and $a\in\mc$ such that $c+\mathfrak{Re}(a)\notin -\mn$, gives us a map:
\begin{equation*}
    \begin{matrix}
        H^a_\alpha ((\caln_V^M\setminus V) \times \mr)\cap E^c_\gamma & \hookrightarrow & H^a_\alpha (\caln_V^M \times \mr)\cap E^c_\gamma\\
        u & \mapsto & \text{\raisebox{-1 pt}{\FiveStar}}_*\left( \widetilde{\text{\raisebox{-1 pt}{\FiveStarCenterOpen}}^*u}\right)
    \end{matrix}
\end{equation*}
where $\widetilde{\text{\raisebox{-1 pt}{\FiveStarCenterOpen}}^*u}$ is the extension of $\text{\raisebox{-1 pt}{\FiveStarCenterOpen}}^*u$ provided by \Cref{th canonical extensions in polar coordinates and weak homogeneity}.

\begin{remark}[Optimality]
    We cannot loosen the condition on $a$ and $c$: it would not work if $c+\mathfrak{Re}(a)\in\rrbracket-\infty,0\rrbracket$. Assume $\caln_V^M= \mr^n\to \{0\}$, take the function 
    \begin{equation*}
        u=\begin{bmatrix}
            (\mr^n\setminus \{0\}) \times \mr^* & \to & \mc\\
            (X,t) & \mapsto & |t|^{\frac{3}{2}}\frac{1}{|tX|^{n+1}}
        \end{bmatrix}.
    \end{equation*}
    If we identify $u$ with the distribution $udXdt$, and with this we have $u\in F_\gamma^{-n-\frac{1}{2}}$, it admits a canonical extension in $E_\gamma^{-n-\frac{1}{2}}$ that is $\frac{-1}{2}$-homogeneous. We will also denote this extension by $u\in C^{-\infty}((\mr^n\setminus \{0\}) \times \mr)$. Now if $u$ admits an $\frac{-1}{2}$-homogeneous extension $\widetilde{u}\in C^{-\infty}(\mr^n\times \mr)$ such that $\widetilde{u}\in E_\beta^{-n-1}$. This would mean that 
    \begin{equation*}
        \frac{1}{|X|^{n+1}}dX \in C^{-\infty}(\mr^n\setminus \{0\})
    \end{equation*}
    always admits an weakly homogeneous (for the dilation action $\mr_+*\curvearrowright \mr^n$) extension of order $-1$, which is false (see \cite{meyer1997wavelets} Chap.2 section 5).
\end{remark}

Now our main result. 

\begin{theorem}[Extension theorem for tempered distributions]\label{theorem extension tempered distributions section 7}\ \\
    Let $a\in\mc$ and let $u\in H^a_\alpha((\caln_V^M\setminus V)\times \mr)$. Recall the decomposition of \Cref{lemma characterization of homogeneous distributions on R+* x N x R} and assume that 
    $w=u|_{\cals_V^M\times \mr} \in S'(\cals_V^M\times \mr \to \cals_V^M)$ then $u$ admits an $a$-homogeneous extension $\widetilde{u}\in S'(\caln_V^M\times \mr \to V)$.
\end{theorem}

\begin{proof}
    Just take $w$ and write it as a finite sum $w=\sum_i w_i$ of weakly-homogeneous distributions thanks to \Cref{Schwartz tempered distrib lemma vector bundle}. Now choose your favorite $\varepsilon\in[0,1[$ such that every weakly-homogeneous summand $w_i$ of the decomposition provided by \Cref{remark epsilon decomposition of Schwartz functions} respects the hypothesis of \Cref{th canonical extensions in polar coordinates and weak homogeneity}. Apply the theorem to each summand $u_i=\int_{s>0} s^{a-1}\gamma^*_s w_i ds$, and the sum of the extensions provides us with and $a$-homogeneous extension $\sum _i \widetilde{u_i}=\widetilde{u} \in C^{-\infty}(\caln_V^M\times \mr)$ of $u$. This extension is $S'(\caln_V^M\times \mr \to \caln_V^M)$ by \Cref{remark weak homogeneity and temperness}. Since $\widetilde{T}$ is $a$-homogeneous one easily show that it implies that $\widetilde{T}\in S'(\caln_V^M\times \mr \to V)$.
\end{proof}

\begin{remark}
    The extension provided by this theorem is not canonical, it relies on the decomposition provided by \Cref{Schwartz tempered distrib lemma vector bundle} which is not canonical.
\end{remark}

Finally we can now answer our original question: given an $a$-homogeneous distribution $u\in C^{-\infty}(\dnc(M,V)\setminus V\times\mr)$ does $u$ have an $a$-homogeneous extension on $\dnc(M,V)$ ? Yes. Is our extension method canonical in some sense? No.

Take $a\in\mc$, take $u\in C^{-\infty}(\dnc(M,V)\setminus (V\times\mr))$. Recall the setting of \Cref{section 4}. One can easily construct a bump function $\chi:  \caln_V^M \to \mr_+$ such that
$\operatorname{supp}(\chi)\subset U$ where $U$ is the domain of the $\operatorname{Exp}$ map. Notice that $\chi$ has proper support for the projection $\caln_V^M \to V$. The smooth function $\operatorname{Exp}_* \chi$ can be extended by $0$ as a smooth function on $M$. Then the function $h$ given by
\begin{equation*}
    h:=\begin{bmatrix}
        \dnc(M,V) & \rightarrow & \mr_+ \\
        (x,t\neq0) & \mapsto & (\operatorname{Exp}_* \chi)(x)\\
        (x,X,0) & \mapsto & 1
    \end{bmatrix}.
\end{equation*}
gives us a smooth bump function around $\caln_V^M\times \{0\} \sqcup V\times \mr^*$ with support in $\Exp(\mathcal{U})$. This function is trivially $0$-homogeneous for the $\alpha$-action. Then $\Exp^* (h\cdot u)$ gives us an $a$-homogeneous distribution on $(\caln_V^M\setminus V)\times\mr$. The extension problem can then be fully answered with our approach. 
Take $v=\Exp^* (h\cdot u)$, thanks to \Cref{lemma characterization of homogeneous distributions on R+* x N x R} write it in polar coordinates as:
\begin{equation*}
    \forall (f,g) \in C^{\infty}_c(\mr_+^*)\times C^{\infty}_c(\cals_V^M\times\mr),\ \dual{v}{f\otimes g} = \int_{\mr_+^*} s^{a-1} f(s) \dual{w}{\gamma_{s^{-1}}^* g} ds
\end{equation*}
with $w \in  C^{-\infty}(\cals_V^M \times \mr)$. Notice that since $\operatorname{supp}(\Exp^* (h\cdot u))\subset \mathcal{U}$ we have that $w$ has proper support for $\cals_V^M \times \mr \to \cals_V^M$. Then chose a partition of unity $(\chi_i)_i$ on $\cals_V^M$ denote $(\boldsymbol{\chi}_i)_i$ the associated partition of unity on the trivial bundle. Then write $w=\sum_i w_i$ where $w_i=\boldsymbol{\chi}_i \cdot w$. Reconstruct the corresponding homogeneous distributions on $(\caln_V^M\setminus V)\times\mr$ by:
\begin{equation*}
    v_i = \text{\raisebox{-1 pt}{\FiveStar}}_*[f\otimes g \mapsto \int_{\mr_+^*} s^{a-1} f(s) \dual{w_i}{\gamma_{s^{-1}}^* g} ds].
\end{equation*}
We have that every $w_i$ is compactly supported, thus a tempered distribution. We can then apply \Cref{theorem extension tempered distributions section 7} to each $v_i$ and get an $a$-homogeneous extension $\sum_i \widetilde{v_i}=\widetilde{\Exp^* (h\cdot u)}\in C^{-\infty}(\caln_V^M \times \mr)$. Then one get an $a$-homogeneous extension of $u$ on $\dnc(M,V)$ with
\begin{equation*}
    \widetilde{u}=\Exp_* \widetilde{\Exp^* (h\cdot u)} + (1-h)\cdot u
\end{equation*}
where $(1-h)\cdot u \in  C^{-\infty}(\dnc(M,V))$ is obtained by trivially extending $(1-h)\cdot u$ by $0$ around $\caln_V^M\times \{0\} \sqcup V\times \mr^*$.

Now notice that this extension is highly non canonical as it depends on the choice of $h$, the choice of the partition of unity $(\chi_i)_i$ on $\cals_V^M$ and the decompositions provided by \Cref{Schwartz tempered distrib lemma vector bundle} for each  $\chi_i$. Two given $a$-homogeneous extensions differ by an $a$-homogeneous extension with support in $V\times\mr$.

\begin{remark}
    This extension theorem can be made canonical in very nice cases.
    \begin{enumerate}
        \item If the exponential map is global i.e. $U=\caln_V^M$ then no need to chose a bump function $h$ and the weak-homogeneity for the $\gamma$ action can be defined intrinsically. Then if $u\in C^{-\infty}(\dnc(M,V)\setminus (V\times\mr))$ happens to be weakly $c$-homogeneous for the $\gamma$ action then one can try to apply \Cref{th canonical extensions in polar coordinates and weak homogeneity}.
        \item If $u\in C^{-\infty}(\dnc(M,V)\setminus (V\times\mr))$ and $\operatorname{supp}(u)\subset\Exp(\mathcal{U})$ the weak-homogeneity for the $\gamma$ action can be defined intrinsically in general. Then if also the projection of $\operatorname{supp}(u)$ on $V$ is compact, one can work to find a relevant "biggest" weak homogeneity degree for $\gamma$ and apply \Cref{th canonical extensions in polar coordinates and weak homogeneity}.
        \item The case of \Cref{th canonical extensions in polar coordinates and weak homogeneity} where $c+\mathfrak{Re}(a)>0$ correspond exactly to \Cref{th locally integrable homogeneous distributions NxR}. This last theorem's assumption are local conditions around $\caln_V^M\times \{0\} \sqcup V\times \mr^*$ and the given extension does not depend on any given bump function around $\caln_V^M\times \{0\} \sqcup V\times \mr^*$.
    \end{enumerate}
\end{remark}

Finally we can characterize the discrepancy in the uniqueness of the extension. For the following proposition we can assume with a partition of unity argument that $V$ is compact. Given $a\in\mc$, denote by $H^a_{\alpha,V\times\mr}(\caln_V^M\times\mr)$ the set of $a$-homogeneous distributions on $\caln_V^M\times\mr$ with support in $V\times\mr$.

\begin{prop}[Characterization of homogeneous distribution supported on $V\times\mr$]\label{characterization of homogeneous distribution supported on VxR}\ \\
    We have a decomposition of the above set by:
    \begin{equation*}
        H^a_{\alpha,V\times\mr}(\caln_V^M\times\mr)=\bigoplus_{b\in\mn}H^a_{\alpha,V\times\mr}(\caln_V^M\times\mr)\cap E_\beta^{-b}.
    \end{equation*}
    Moreover we have an explicit description of the summand given by four cases.
    \begin{enumerate}
        \item If $a+b\notin\mz$, then:
        \begin{align*}
            H^a_{\alpha,V\times\mr}(\caln_V^M\times\mr)\cap E_\beta^{-b} = & \{ P_-\nu_-\otimes t_-^{-a-b-1}  + P_+\nu_+\otimes t_+^{-a-b-1}  \ | \\
            & \quad P_-,P_+\in \operatorname{HomDiffop}^b(\caln_V^M),\ \nu_-,\nu_+\in C^{-\infty}(V)\}.
        \end{align*}
        \item If $a+b\in -\mn^*$, then $H^a_{\alpha,V\times\mr}(\caln_V^M\times\mr)\cap E_\beta^{-b}$ is given by:
        \begin{align*}
            & \{ P_-\nu_-\otimes t_-^{-a-b-1}  + P_0\nu_0\otimes\partial_t^{-a-b}\delta_0 + P_+\nu_+\otimes t_+^{-a-b-1} \ | \\
            & \quad P_-,P_0,P_+\in \operatorname{HomDiffop}^b(\caln_V^M),\ \nu_-,\nu_0,\nu_+\in C^{-\infty}(V)\}.
        \end{align*}
        \item If $a+b=0$, then:
        \begin{align*}
            H^a_{\alpha,V\times\mr}(\caln_V^M\times\mr)\cap E_\beta^{-b}= & \{ -P\nu\otimes t_-^{-1}  + P_0\nu_0\otimes\delta_0 + P\nu\otimes t_+^{-1}  \ | \\
            & \quad P,P_0\in \operatorname{HomDiffop}^b(\caln_V^M),\ \nu,\nu_0\in C^{-\infty}(V)\}.
        \end{align*}
        \item If $a+b\in\mn^*$, then:
        \begin{align*}
            H^a_{\alpha,V\times\mr}(\caln_V^M\times\mr)\cap E_\beta^{-b} =& \{ -P\nu \otimes t_-^{-a-b-1}  + P\nu\otimes t_+^{-a-b-1}  \ | \\
            & \quad P\in \operatorname{HomDiffop}^b(\caln_V^M),\  \nu\in C^{-\infty}(V)\}.
        \end{align*}
    \end{enumerate}
\end{prop}

\begin{proof}
    A distribution on $\caln_V^M\times\mr$ supported in $V\times\mr$ is a sum of order $b\in\mn$ normal derivatives of distributions on $V\times\mr$. Thus the direct sum. Now given $b\in\mn$ the study of the corresponding summand boils down to the study of $(-a-b)$-homogeneous distributions for the $\gamma$ action. Such a study can be conducted exhaustively on $V\times\mr^*$ thanks to \Cref{th homogeneous functions and the trivial action}. And then we can compute the obstructions for such distributions to be extended on $V\times\mr$, by adding the suitable derivatives of distributions supported on $V\times\{0\}$, we retrieve all $(-a-b)$-homogeneous distributions for the $\gamma$ action on $V\times\mr$. Thus retrieving all the distributions of $H^a_{\alpha,V\times\mr}(\caln_V^M\times\mr)\cap E_\beta^{-b}$.
\end{proof}

Now a final discussion would be to link these results to the ones of \cite{vanerp2017groupoid}. The missing condition in our work is the notion of transversality of \cite{androulidakis2009}. In our setting we only need to be interested in the transversality with respect to $\pi: \caln_V^M\times\mr \to V\times\mr$. Let us denote by $C^{-\infty}_\pi(\caln_V^M\times\mr)$ the set of $\pi$-transversal distributions on $\caln_V^M\times\mr$. We can modify the previous result by adding this assumption and obtain the following.

\begin{corollary}[Characterization of homogeneous $\pi$-transversal distributions supported on $V\times\mr$]\label{characterization of homogeneous pi-transversal distribution supported on VxR}\ \\
    Let $a\in\mc$, the set $H^a_{\alpha,V\times\mr}(\caln_V^M\times\mr)\cap C^{-\infty}_\pi(\caln_V^M\times\mr)$ admits a decomposition given by:
    \begin{equation*}
        H^a_{\alpha,V\times\mr}(\caln_V^M\times\mr)\cap C^{-\infty}_\pi(\caln_V^M\times\mr) =\bigoplus_{b\in\mn}H^a_{\alpha,V\times\mr}(\caln_V^M\times\mr)\cap C^{-\infty}_\pi(\caln_V^M\times\mr)\cap E_\beta^{-b}.
    \end{equation*}
    Moreover we have an explicit description of the summand. It is non zero iff $a+b\in -\mn^*$. In this case $H^a_{\alpha,V\times\mr}(\caln_V^M\times\mr)\cap C^{-\infty}_\pi(\caln_V^M\times\mr)\cap E_\beta^{-b}$ is given by:
    \begin{align*}
        & \{ (-1)^{-a-b-1} P\nu\otimes t_-^{-a-b-1}  + P\nu\otimes t_+^{-a-b-1} |\  P\in \operatorname{HomDiffop}^b(\caln_V^M),\ \nu\in C^{\infty}(V,|\Lambda|^1_V)\}.
    \end{align*}
    
\end{corollary}

\begin{remark}
    If we apply this result to $\dnc(M\times M,\Delta M)=\mathbb{T}M$ as in \cite{vanerp2017groupoid}, since we can work on the exponential charts for $\pi$ to coincide with the range or source maps of $\mathbb{T}M$, the previous set is exactly $\mathbb{\Psi}^a\cap  \cale_r'(\mathbb{T}M)^{(0)}$. It is the set of differential operators of \cite{vanerp2017groupoid}, section 10.
\end{remark}

That being said observe the following regarding \Cref{theorem extension tempered distributions section 7}:
\begin{enumerate}
    \item The $\text{\raisebox{-1 pt}{\FiveStarCenterOpen}}^*$ and the $\text{\raisebox{-1 pt}{\FiveStar}}_*$ maps preserve $\pi$-transversality (denote also by $\pi: \mr_+\times\cals_V^M\times\mr \to V\times\mr$ the projection on the base in polar coordinates).
    \item Given a distribution $u$ such that \Cref{th canonical extensions in polar coordinates and weak homogeneity} holds, thanks to the explicit formulas provided, both $u$ and $\widetilde{u}$ are $\pi$-transversal if and only if $w$ is transversal for the projection $\cals_V^M\times\mr \to V\times\mr$.
\end{enumerate}
We have then that our extension result of \Cref{th canonical extensions in polar coordinates and weak homogeneity} on $\caln_V^M\times\mr$ preserves $\pi$-transversality. Moreover, thanks to the formula, if we suppose $u$ to be smooth then its extension $\widetilde{u}$ provided by the theorem has its wavefront set contained in ${\caln_V^M}^*\times\mr$ the conormal bundle over $V\times\mr$. We can then describe the discrepancy in the uniqueness of the extension in the $\pi$-transversal case.

\begin{remark}[Weak homogeneity and kernels of classical pseudodifferential operators]
    Given a classical pseudodifferential operator in the sense of \cite{vanerp2017groupoid} $\mathbb{P}\in\mathbb{\Psi}$, given $b\in\mr$ such that $\mathbb{P}\in O^{-b}_\beta$. The classical operator with kernel $\mathbb{P}|_{t=1}$ is of order $b$. This means that when $W\to V=TM\times\mr \to M\times\mr$ the obstruction to the extension result of \Cref{th canonical extensions in polar coordinates and weak homogeneity} is related to the polyhomogeneous extension of $\mathbb{P}|_{t=1}$ near the diagonal $\Delta M \subset M\times M$, this was actually the initial motivation of the paper.
\end{remark}

\begin{remark}[Extension and pseudodifferential operators]\ \\
    Take a smooth, $a$-homogeneous density $f\in C^{\infty}(\caln^M_V\times\mr \setminus V\times\mr, |\Lambda|^1\operatorname{ker}d\pi)$, thanks to what we saw, we can find an $a$-homogeneous extension $\widetilde{f}\in C^{-\infty}_\pi(\caln^M_V\times\mr)$.
    \begin{itemize}
        \item If $a\notin-\mn$ such a $\pi$-transversal extension is unique thanks to \Cref{characterization of homogeneous pi-transversal distribution supported on VxR}.
        \item Given any bump function $\chi\in C^{\infty}_{\operatorname{fc}}(\caln^M_V\times\mr)$ around $V\times\mr$ with $\pi$-proper support, we obtain $\chi \widetilde{f}\in\mathbb{\Psi}^a$ when $\dnc(M,V)=\mathbb{T}M$.
    \end{itemize}
\end{remark}

The idea of treating pseudodifferential operators by only looking at exact homogeneity in fact already appears in the use of Proposition 43 in \cite{vanerp2017groupoid}. The proposition tells us that there is a one to one correspondence between $\faktor{\mathbb{\Psi}^a(M)}{\mathbb{\Psi}^{-\infty}(M)}$ and $C^\infty (T^*M \times\mr\setminus M\times\{0,0\})\cap H^a(T^*M \times\mr\setminus M\times\{0,0\})$, where homogeneity is with respect to the dilation action on the vector bundle $T^*M \times\mr\to M$. An explicit and simple construction of the correspondence is provided in \cite{cren2023filtered}. However because of singularities of such homogeneous $A\in C^\infty (T^*M \times\mr\setminus M\times\{0,0\})$ around $M\times\{0,0\}$, given $P\in \Psi^a(M)$, we cannot find in general $\mathbb{P}\in H^a_{\alpha}(\mathbb{T}M)\cap C^{-\infty}_r(\mathbb{T}M)$ such that $\mathbb{P}|_{t=1}=P$. This would mean that the Wodzicki residue of $P$ and the Kontsevitch-Vishik trace of $P$ (when defined) always vanish thanks to \cite{mohsen2024wodzicki} which is absurd.

\printbibliography

\end{document}